\documentclass[10pt,a4paper]{article}
\usepackage{amsmath,amsxtra,amsthm,amssymb,makeidx,graphics,mathrsfs}
\usepackage{graphicx}
\usepackage{enumerate}
\usepackage{stmaryrd}
\usepackage[all]{xy}
\usepackage{float}
\usepackage{mathtools}
\usepackage{color,tocloft}
\usepackage{cite}
\usepackage{hyperref}
\hypersetup{
  bookmarksopen = false,
  pageanchor=true,
  pdfcreator={Aaron Chan},
  pdfauthor={Aaron Chan},
  pdfstartview = FitW
}

\numberwithin{equation}{section}

\theoremstyle{plain}
\newtheorem{theorem}{Theorem}

\newtheorem{prop}[theorem]{Proposition}
\newtheorem{lemma}[theorem]{Lemma}

\newtheorem{definition}[theorem]{Definition}

\theoremstyle{remark}
\newtheorem*{remark}{Remark}

\makeatletter
\def\uddots{\mathinner{\mkern1mu\raise\p@
\vbox{\kern7\p@\hbox{.}}\mkern2mu
\raise4\p@\hbox{.}\mkern2mu\raise7\p@\hbox{.}\mkern1mu}}
\makeatother

\mathchardef\hy="2D
\newcommand{\op}{\mathrm{op}}

\newcommand{\integer}{\mathbb{Z}}

\newcommand{\isomto}{\xrightarrow{\sim}}
\newcommand{\isom}{\cong}

\newcommand{\imply}{\Rightarrow}

\newcommand{\id}{\mathrm{id}}
\newcommand{\Ext}{\mathrm{Ext}}

\newcommand{\Hom}{\mathrm{Hom}}
\newcommand{\End}{\mathrm{End}}

\newcommand{\lmod}[1]{#1\mathrm{\hy mod}}
\newcommand{\gr}[1]{#1\mathrm{\hy gr}}
\newcommand{\catO}{\mathcal{O}}

\newcommand{\std}{\Delta}
\newcommand{\costd}{\nabla}
\newcommand{\projresstd}{\widetilde{\Delta}}
\newcommand{\injrescostd}{\widetilde{\nabla}}

\newcommand{\Rad}{\mathrm{Rad}}
\newcommand{\sn}{\mathfrak{S}}
\newcommand{\add}{\mathrm{add}}
\newcommand{\blambda}{{\boldsymbol{\lambda}}}
\newcommand{\bmu}{\boldsymbol{\mu}}
\newcommand{\tiltres}{\mathcal{T}}
\newcommand{\projresX}{\widetilde{X}}

\begin{document}
\title{Some Homological Properties of Tensor and Wreath Products of Quasi-hereditary Algebras}
\author{Aaron Chan}
\maketitle
\begin{abstract}
	We show that taking the wreath product of a quasi-hereditary algebra with symmetric group inherits several homological properties of the original algebra, namely BGG duality, standard Koszulity, balancedness as well as a condition which makes the Ext-algebra of its standard modules a Koszul algebra.
\end{abstract}

\section{Introduction}
Let $A$ be a unital finite dimensional quasi-hereditary $F$-algebra with respect to the weight poset $(I,\leq)$, where $F$ is an algebraically closed field.  $I$ is then the indexing set for the isomorphism classes of six families of modules, namely the standard modules $\{\std(i)\}$, costandard modules $\{\costd(i)\}$, tilting modules $\{T(i)\}$, projectives $\{P(i)\}$, injectives $\{Q(i)\}$ and simples $\{L(i)\}$.  We will call these six families of modules the \textit{structural families} or \textit{structural modules} of $A$.  We use the symbol $X$, for $X\in\{P,Q,L,T,\std,\costd\}$, to denote the direct sum of representatives of all the isomorphism classes of the corresponding special family, i.e. $X=\bigoplus_{i\in I}X(i)$.

We assume $A=\bigoplus_{n\in \integer_{\geq 0}}A_n$, i.e. it is graded by $\integer_{\geq 0}$.  Since we use $\integer$-grading throughout, for convenience ``graded" will always mean ``$\integer$-graded".  The category of locally finite dimensional graded modules is denoted $\gr{A}$, note that $\Hom_{\gr{A}}(M,N)$ consists of maps from $M$ to $N$ which are homogenous of degree 0.  For $M=\bigoplus_{n\in\integer}M_n\in\gr{A}$, we let $M\langle k\rangle$ denotes the grading shift such that $(M\langle k\rangle)_n = M_{k+n}$.

When $A$ is graded with $A_0\isom A/\Rad A$, the structural modules of $A$ have ``standard" graded lifts as follows.  As any simple $A$-module $L(i)$ can be identified with summands of $A_0$, the standard graded lift of $L(i)$ is concentrated in degree 0, i.e. $L(i) = L(i)_0$.  Recall the natural morphisms on the structural $A$-modules:
\begin{equation}\label{eqn-strucmap}
\xymatrix@R=10pt{
P(i) \ar@{->>}[r] & \std(i) \ar@{->>}[r]\ar@{^{(}->}[rd]& L(i)\ar@{^{(}->}[r]& \costd(i)\ar@{^{(}->}[r] & Q(i)\\
     &         & T(i) \ar@{->>}[ru]&           &
}
\end{equation}
The standard graded lifts of the structural $A$-module are chosen such that all the maps above live in $\gr{A}$.  Following the notion in the classical Koszul theory, a complex of structural modules $\mathcal{X}^\bullet$ given by
\[
\cdots \to \mathcal{X}^{n-1}\xrightarrow{d^{n-1}} \mathcal{X}^n \xrightarrow{d^n} \mathcal{X}^{n+1} \to\cdots
\]
is said to be \emph{linear} if all indecomposable summands of $\mathcal{X}^n$ are isomorphic to $X(i)\langle n\rangle$ for some $i\in I$.  \textcolor{white}{)))}

For $X\in\{P,Q,L,T,\std,\costd\}$, we denote by $A^X$ the Yoneda Ext-algebra $\Ext^\bullet_A(X,X)^\op$ of a structural family of $A$-modules.  Understanding the structure of the Ext-algebra $A^X$ for a given algebra $A$ is then a natural and interesting question to ask.  In the case of $X\in\{P,Q\}$, one gets the basic algebra associated to $A$.  If $X=T$, then one gets the \emph{Ringel dual} $A'$ of $A$, this is a quasi-hereditary algebra with respect to $(I,\leq^\op)$.  In these cases, we get a derived equivalence between $A$ and $A^X$ given by the $A\hy A^X$ (two-sided) tilting complex $X$.  However, for $X\in \{ L,\std,\costd\}$, the properties of $A^X$ are generally much more obscure.  One then has to restrict to subclasses of quasi-hereditary algebras which exhibit nice homological properties.

A quasi-hereditary algebra $A$ is \emph{standard Koszul} if there is a grading on $A$, so that for each $i\in I$, a minimal graded projective resolution $\projresstd(i)^\bullet$ of standard modules $\std(i)$ and a minimal graded injective coresolution $\injrescostd(i)^\bullet$ of $\costd(i)$ are both linear.  This notion was first established in \cite{ADL}, where they have proved that $A$ is then a Koszul algebra and the Koszul dual $A^! (=A^L)$ of $A$ is quasi-hereditary with respect to $(I,\leq^\op)$.  Also note that if $A$ is quasi-hereditary and Koszul, then the Koszul grading will satisfy the assumption for standard Koszulity.  By \cite{BGS}, one then gets a derived equivalence between $A$ and $A^!$; also see \cite{Mad1} on unifying such derived equivalence with the one arising from Ringel duality.

Also recall from \cite{Irv} that $A$ is a \emph{BGG algebra} if it is quasi-hereditary and there is a duality functor on the category of finitely generated modules $\lmod{A}$, i.e. a contravariant exact functor $\delta$ on $\lmod{A}$ such that $\delta^2 \isom \id_{\lmod{A}}$, and for all $i\in I$, we have $\delta L(i)\isom L(i)$.  When $A$ is a BGG algebra, $\delta P(i)\isom Q(i)$ and $\delta \std(i)\isom \costd(i)$, see \cite{Irv}.  In particular, linearity on the resolutions of standard modules will suffice to show standard Koszulity.

Following \cite{Maz2}, $A$ is said to be \emph{balanced} if $A$ is standard Koszul, and for each $i\in I$, a minimal graded tilting coresolution $\tiltres_\std(i)^\bullet$ of the standard module $\std(i)$ and a minimal graded tilting resolultion $\tiltres_\costd(i)^\bullet$ of the costandard module $\costd(i)$ are both linear.  In this case, the Ringel dual of $A$ is also Koszul, and $(A^T)^L \isom (A^L)^T$, see \cite{Maz2}.

Much less is known about $A^\std$ and $A^\costd$ in general.  The first property one can say is that they are quasi-hereditary with respect to both $(I,\leq)$ and $(I,\leq^\op)$ since they are directed algebras.  It is then desirable to ask for a derived equivalence between $A$ and $A^\std$ as the homological algebra for a directed algebra is usually relatively easier to understand.  Madsen approached this problem using generalised Koszul duality \cite{Mad2}, and we will come to this soon.  Another natural problem one would like to know is that when will $A^\std$ be Koszul.

Drozd-Mazorchuk showed that if a graded quasi-hereditary algebra $A$ is equipped with a function $h:I\to \{0,1,\ldots,n\}$, where $n$ is a natural number, and satisfies the following four conditions,
\vspace{-5mm}\begin{enumerate}[(I)]
\item $\tiltres_\std(i)^k\in \add\left(\oplus_{j: h(j)=h(i)-k} T(j)\langle k\rangle\right)$ for all $k\geq 0$;
\item $\tiltres_\costd(i)^k\in \add\left(\oplus_{j: h(j)=h(i)+k} T(j)\langle k\rangle\right)$ for all $k\leq 0$;
\item $\projresstd(i)^k\in \add\left(\oplus_{j: h(j)=h(i)-k} P(j)\langle k\rangle\right)$ for all $k\leq 0$;
\item $\injrescostd(i)^k\in \add\left(\oplus_{j: h(j)=h(i)+k} Q(j)\langle k\rangle\right)$ for all $k\geq 0$.
\end{enumerate}
then $A^\std$ is Koszul.  More explicitly,

\begin{theorem}[Drozd-Mazorchuk \cite{DM}]
Let $A$ be a quasi-hereditary algebra equipped with a function satisfying conditions (I)-(IV).  Then $A^\std$ is Koszul, with Koszul dual the Ext-algebra of costandard $A$-modules, i.e. $(A^\std)^! \isom (A^\costd)^\op$.
\end{theorem}
In fact, the original theorem contains more information, but we will omit this for now, since the results of Madsen \cite{Mad1, Mad2} also recover the same information.  Madsen's works use the theory of $T$-Koszulity, which was first introduced in \cite{GRS}, and combining with inspiration from \cite{DM}, to unify the theory of $A^\std$ (or $A^\costd$) with that of $A^X$ for $X\in\{P,Q,T,L\}$ when $A$ satisfies the following condition:
\begin{definition}
Let $(A,(I,\leq))$ be a standard Koszul BGG algebra.  Then we say $A$ \emph{satisfies condition (H)} when $A$ is equipped with a function $h:I\to \{0,1,\ldots,n\}$ such that, if the $k$-th radical layer of $\std(x)$ contains $L(y)$, then $h(y)=h(x)-k$.  That is to say:
\[
[\std(x) : L(y)\langle l\rangle]_q \neq 0 \;\;\imply\;\; h(y)=h(x)-l
\]
Note the multiplicity bracket and modules here are graded by the Koszul grading.
\end{definition}

It was originally proved in \cite{DM} that if the algebra associated to a block of category $\catO$ of a complex semisimple Lie algebra satisfies (H), then the block satisfies condition (H), which gives a function satisfying conditions (I)-(IV).  In \cite{Mad2}, Madsen relaxes this  to standard Koszul BGG algebras (rather than just specific blocks of category $\catO$); one can then able to reproduce most of the results in \cite{DM} through the use of $T$-Koszulity.  Moreover, one can now get a derived equivalence of graded $A$ and $A^\std$-modules, with a grading different from the Koszul grading and homological grading, termed as $\std$-grading by Madsen.  This grading has actually been seen ``in disguise" in \cite{DM}, and also has appeared in other investigations of $A^\std$ such as \cite{MT}.

\begin{theorem}[Madsen, \cite{Mad1},\cite{Mad2}]
Let $A$ be a standard Koszul BGG algebra satisfying (H).  Then \vspace{-5mm}
\begin{enumerate}
\item $A$ satisfies conditions (I)-(IV), hence $A$ is balanced and $A^\std$ is Koszul.
\item There is a $\std$-grading on $A$, i.e. $A$ is positively graded with $A_0 = \std$ and $A_i \in \add(\std)\langle i\rangle$ for all $i>0$.  Note the shift $\langle i\rangle$ here is on the Koszul grading.
\item Taking $T=A_0$ in the $\std$-grading, then $T$ satisfies the axioms of $T$-Koszulity.  In particular
\begin{enumerate} 
\item $D\std:= \Hom_F(\std,F)$ is an $A^\std$-module, and $A \isom  (A^\std)^{D\std}$
\item There is a graded derived equivalence:
\[
D^b( \gr{A}) \xrightarrow{\stackrel{R\Hom_A(\std,-)}{\sim}} D^b(\gr{A^\std})
\]
which sends costandard $A$-modules to simple $A^\std$-modules.
\end{enumerate}
\end{enumerate}
\end{theorem}
(3)(b) above now gives a rigorous meaning to the idea of Drozd-Mazorchuk (originated from Ovsienko) that costandard $A$-modules can be ``aligned" in such a way that they are simple $A^\std$-modules, inducing Koszulity of $A^\std$.

So far, we have introduced different subclasses of quasi-hereditary algebras which have nice homological properties, and these nice properties give information on how the Ext-algebras of the structural families behave.  In this article, we show that these nice properties can be carried over to tensor products of such algebras, as well as wreath products of such algebras with the symmetric group.

Throughout the article, any tensor product of vector spaces $\otimes$ without a subscript is the tensor product over the underlying field $F$.  Given an algebra $A$ and $w\in\integer_{>0}$, there is a natural action of the symmetric group $\sn_w$ on the tensor product $A^{\otimes w}$ by permuting components.  The wreath product of $A$ with the symmetric group $\sn_w$ is the vector space $A^{[w]} := A^{\otimes w}\otimes F\sn_w$, with multiplication
\[
(a_1\otimes\cdots\otimes a_w\otimes\sigma)(b_1\otimes \cdots\otimes \tau)=a_1b_{\sigma^{-1}(1)} \otimes\cdots \otimes a_wb_{\sigma^{-1}(w)}\otimes \sigma\tau
\]
for $\sigma,\tau\in\sn_n$.  We will simply call such an algebra the wreath product algebra or wreath product of $A$.

In the representation theory of symmetric groups and their quasi-hereditary covers (Schur algebras) over prime characteristic, the ``complexity" of blocks are measured by weights $w\in\integer_{>0}$ (not to be confused with the notion of weights in highest weight theory).  The weight zero blocks are the semisimple blocks and the weight one blocks are the Brauer tree algebras, and their quasi-hereditary covers, the zigzag algebras.  These algebras have been thoroughly studied throughout the literature.  For each given weight $w$, with $w>$char$F$, there is a special kind of block, called the Rouquier block or RoCK block, which is the simplest block to understand in terms of its homological behaviour.  The reason for this is because the RoCK block of weight $w$ is Morita equivalent to the wreath product of the weight one block with $\sn_w$; similar situations also occur in other areas of ``type A representation theory", see for example \cite{CT2}.  This particular example is our motivation to show that wreath product algebras inherits nice homological properties of the original algebra.  Since we will need results from \cite{CT} in our exposition, we will impose the extra condition that $w!$ is invertible in the field $F$ when we study wreath product algebras in Section 3.

The rest of this article consists of two sections.  The first surveys some results on tensor products of quasi-hereditary algebras with nice homological properties (BGG and/or standard Koszul and/or balanced and/or condition (H)).  We will also show that taking the Ext-algebra of a structural families over the tensor product of algebras is the same as taking the tensor products of the Ext-algebras of the structural families.  Most of these results are folklore, but we have yet to find good enough references for them, we will thus include some simplistic proof for each of them.  In the second section, we show the analogous results for wreath product algebras, using roughly the same ideas from the proofs in Section 2.

\section{Tensor product algebras}
Let $A_1, A_2$ be quasi-hereditary algebras and $(I_1,\leq_1), (I_2,\leq_2)$ be the respective weight posets.  The tensor product algebra $A:=A_1\otimes A_2$ is then quasi-hereditary with respect to $(I:=I_1\times I_2,\leq)$ where the partial order $\leq$ is defined by: $
(x_1,x_2) \leq (y_1,y_2)$ if $x_k \leq y_k$ for $k=1,2$.  This comes from the fact that each structural $A$-module is the tensor product of structural modules of $A_1$ and $A_2$, namely $X(x_1,x_2)=X_{A_1}(x_1)\otimes X_{A_2}(x_2)$ for $X\in\{P,Q,L,\std,\costd,T\}$ and all $(x_1,x_2)\in I$.  For simplicity, we denote structural $A_i$-module $X_{A_i}(x)$ by $X_i(x)$ for $i=1,2$.

\begin{prop}\label{prop-tenExtCommute}
Tensoring and taking the Ext-algebra of structural modules are commuting operations on algebras. i.e. $(A_1\otimes A_2)^X \isom A_1^{X_1}\otimes A_2^{X_2}$
\end{prop}
\begin{proof}
Since $X_A\isom X_1\otimes X_2$ for $X\in\{P,Q,L,\std,\costd,T\}$.  It then follows from a well-known folklore that the Ext-algebra of the tensor product of modules is isomorphic (as an algebra) to the tensor product of Ext-algebras.  The closest reference we can find is the generalisation of this result in \cite[Theorem 3.7]{BO}.
\end{proof}

We first show that BGG duality can be induced naturally:
\begin{lemma}\label{lem-dualOnTensor}
If $A_1, A_2$ are BGG algebras, then so is $A_1\otimes A_2$.
\end{lemma}
\begin{proof}
From \cite[Prop 2.1]{CPS}, a duality functor (not necessarily fixing simple modules) corresponds to an anti-automorphism $\iota_i$ of $A_i$ such that $\iota_i^2$ is an inner automorphism $\alpha_i$ of $A_i$.  We see that $\iota := \iota_1\otimes\iota_2$ is an anti-automorphism on $A_1\otimes A_2$ such that $\iota^2 = \alpha_1\otimes \alpha_2$ which is also an inner automorphism of $A$.  Consequently, $\iota$ induces a duality functor $\delta$, which maps $M\in \lmod{A}$ to the vector space $M^*=\Hom_F(M,F)$, with the $A$-action given by $a\cdot f(m)=f(\iota(a)m)$.  In particular, for finite dimensional modules $M_i \in \lmod{A_i}$ for $i=1,2$, since $(M_1\otimes M_2)^*\isom M_1^*\otimes M_2^*$, we have $\delta(M_1\otimes M_2) \isom \delta_1M_1\otimes \delta_2M_2$, where $\delta_i$ is the BGG duality functor on $\lmod{A_i}$.  Since $\delta_iL_i(x)\isom L_i(x)$ for all simple $A_i$-module $L_i(x)$, it follows that the duality on $A_1\otimes A_2$ fixes simple $A_1\otimes A_2$-modules.
\end{proof}

\begin{lemma}\label{lem-balOnTensor}
If $A_1,A_2$ are standard Koszul (resp. balanced) algebras, then so is $A_1\otimes A_2$.
\end{lemma}
\begin{proof}
We use the fact that the total complex of the tensor product of (graded) projective resolutions is the projective resolution of the corresponding tensor product of modules, see for example \cite[Lemma 3.6]{BO}.  This tensor product of resolutions also preserves minimality and linearity.  Applying the dual argument on the injective coresolution, we have the claim for standard Koszulity.

For (graded) tilting (co)resolutions, one just does the same trick.  Using the fact that $T_1(x_1)\otimes T_2(x_2) = T(x_1,x_2)$ is a tilting $A_1\otimes A_2$-module, the tensor product of the tilting (co)resolutions will then be the tilting (co)resolution of the tensor product of the modules, which also preserves minimality and linearity.  Hence the claim for balancedness.
\end{proof}

\begin{prop}
Let $A_1, A_2$ be standard Koszul BGG algebra satisfying condition (H), then so is the tensor product algebra $A_1\otimes A_2$.
\end{prop}
\begin{proof}
The grading of $A=A_1\otimes A_2$ comes naturally from the usual grading on tensor products, Koszulity for this grading follows from Lemma \ref{lem-balOnTensor} above.  BGG duality follows from Lemma \ref{lem-dualOnTensor}.  Given functions $h_j:I_j\to \{0,\ldots,n_i\}$ of $A_j$'s so that $A_j$'s satisfy condition (H), we define the required function $h:I \to \{ 0,1,\ldots,n=n_1+n_2\}$ as follows
\begin{eqnarray*}
h: I & \to & \{0,1,\ldots,n\}\\
 (x_1,x_2) &\mapsto & h_1(x_1)+h_2(x_2)
\end{eqnarray*}
Examining the (Koszul) graded structure of $\std(x)$ closely, its $l$-th graded piece is given by:
\[
(\std(x))_l = \bigoplus_{l_1+l_2 = l} (\std_1(x_1))_{l_1} \otimes(\std_2(x_2))_{l_2}
\]
Therefore, when $[\std(x):L(y)\langle l\rangle]_q\neq 0$, there are some $l_1, l_2$ such that $[\std_j(x_j):L(y_j)\langle l_j\rangle]_q\neq 0$ for $j=1,2$.  Since the $A_j$'s satisfy condition (H) with respect to the $h_j$'s, we have
\begin{eqnarray*}
h_j(x_j) - h_j(y_j) &=& l_j \qquad\qquad\qquad  \mbox{for } j=1,2\\
\imply\quad (h_1(x_1)+h_2(x_2)) - (h_1(y_1)+h_2(y_2)) &=& l_1+l_2 \;\;=\;\; l \\
\imply\hspace*{2.3cm} h(x) - h(y) &=& l
\end{eqnarray*}
and the condition (H) is satisfied.
\end{proof}

By induction, one shows that the above results extend to all finite tensor products of quasi-hereditary algebras.

\section{Wreath product algebras}
Given any (graded) $A$-module $M$, one can take the wreath product of $M$ with symmetric group, i.e. $M^{[w]}:=M^{\otimes w}\otimes F\sn_w$, to get an $A^{[w]}$-module, with the obvious action.  Note that if $M$ is graded, then so is $M^{[w]}$, where the induced grading is given by the usual $\integer$-grading on tensor product of graded modules and putting the component $F\sn_w$ in degree zero.  Wreath product preserves many nice properties of an algebra, its modules, and its complexes (see later paragraphs).  We start by collecting some results from \cite{CT} and \cite[Section 2]{CLS} that will be useful to us.  We remind the reader again that throughout this section, $w$ is a positive integer with $w!$ invertible in the underlying field $F$.  For convenience, we simply say ``wreathing an object" instead of ``taking the wreath product of an object with the symmetric group".

Let $\{X(i)|i\in I\}$ be a (structural) family of $A$-modules, and the cardinality of $I$ be $n\in\mathbb{N}$.  Then there is a family of  $A^{[w]}$-modules, $\{X(\blambda)|\blambda\in \Lambda_w^I\}$, which are indexed by the set of $I$-tuples of partitions such that the sum of the size of the entries is $w$, i.e.
\[
\Lambda_w^I :=\left\{ \blambda = (\lambda^{(1)},\ldots,\lambda^{(n)}) \;\left|\;\; \lambda^{(i)}\vdash \omega_i \mbox{ and }  \sum_{i\in I}\omega_i = w\right.\right\}
\]

\begin{lemma}
For $X\in \{ P,Q,L, \std,\costd, T\}$, the family $\{ X(\blambda) | \blambda \in \Lambda_w^I\}$ is the structural family of $A^{[w]}$-modules.
\end{lemma}
\begin{proof}
For each $\blambda \in \Lambda_w^I$, the $A^{[w]}$-module $X(\blambda)$ occur as indecomposable summands of wreathing the direct sum of structural modules \cite[Lemma 3.8]{CT}:
\[
\left(\bigoplus_{i\in I}X(i)\right)^{[w]} \isom \bigoplus_{\blambda\in\Lambda_w^I} X(\blambda)^{\oplus m(\blambda)}
\]
for some $m(\blambda) \in \mathbb{N}$ which depends only on $\blambda$.  For $X\in\{P,Q,L,\std,\costd\}$, \cite[Lemma 3.8, 3.9, Section 6]{CT} already showed that the new family $\{ X(\blambda)\}$ indeed is the corresponding structural family.  Also from \cite[Section 4]{CT}, one can see that the induced family $\{T(\blambda)|\blambda \in \Lambda_w^I\}$ from the family of tilting $A$-modules is filtered by $\{\std(\blambda)\}$ as well as $\{\costd(\blambda)\}$.  Hence $\{ T(\blambda)\}$ is indeed the family of tilting $A^{[w]}$-modules.
\end{proof}

This construction of a new family of objects can also be applied to maps.  Given two families $\{X(i)\}$, $\{Y(i)\}$ of $A$-modules indexed by $i\in I$, and a family of maps $\{ f_i: X(i)\to Y(i) | i\in I\}$, then there is a family of maps of $A^{[w]}$-modules $\{ f_\blambda :X(\blambda)\to Y(\blambda)| \blambda\in\Lambda_w^I\}$, see \cite[3.9(2)]{CT}.  In particular, if $\{f_i\}$ is a family of one of the structural maps appeared in (\ref{eqn-strucmap}), then $\{f_\blambda\}$ is the corresponding family of structural maps.

From now on, we assume all the modules are graded.  As mentioned in the first paragraph of this section, the wreathing construction applies to bounded complexes of (graded) $A$-modules as follows.  Let $C^\bullet$ be a bounded complex of finite dimensional $A$-modules with differential $d$, then there is a differential on the $A^{\otimes w}$-complex $C^{\otimes w \bullet}$, which can be written as
\[
\sum_{a+b=w-1} 1^{\otimes a} \otimes d \otimes 1^{\otimes b}.
\]
The differential on the $A^{[w]}$-complex $C^{[w]\bullet}$ is given by tensoring the above differential with $1_{F\sn_w}$.  One can also wreath a chain map of $A$-complexes, which will consequently make $(-)^{[w]}$ a functor that preserves homotopy \cite[Lemma 2.4]{CLS}.  In another words, one can regard $(-)^{[w]}$ as a functor from the bounded homotopy category $K^b(\gr{A})$ to the bounded homotopy category $K^b(\gr{A^{[w]}})$.  Note that this functor is polynomial, non-linear, non-additive, and non-triangulated (or non-exact on the full subcategory of modules), see \cite[Section 2]{CLS}.  On the other hand, the wreathing functor preserves monomorphisms and epimorphisms on modules.  We also have a stronger result:

\begin{lemma}\label{lem-wrHomology}
Let $C^\bullet$ be a bounded complex of finite dimensional $A$-modules.  Then the homology of wreathing $C^\bullet$ is the isomorphic to wreathing the homology of $C^\bullet$. i.e. $H(C^{[w]\bullet}) \isom H(C^\bullet)^{[w]}$.  In particular, $f^{[w]}$ is a quasi-isomorphism in $K^b(\gr{A^{[w]}})$ if $f$ is a quasi-isomorphism in $K^b(\gr{A})$.
\end{lemma}
\begin{proof}
Note that for each $n\in\integer$, $C^n = H^n \oplus B^n \oplus L^n$ as a $F$-vector spaces, where $H^n = H^n(C^\bullet)$, $\ker d_C^n = Z^n(C^\bullet) = H^n\oplus B^n$.  In particular, $d^n(C^n) = d^n(L^n) = B^{n+1}$.  We write $(F\to F)^{\oplus \dim L^n}$ to represent the exact part $L^n\xrightarrow{d^n} B^{n+1}$ of the complex.  So we can write $C^\bullet$ as a complex of $F$-vector spaces as follows
\[
C^\bullet = \underbrace{\left(\bigoplus_{n\in \integer} H^n\right)}_{H^\bullet(C^\bullet)} \oplus  \underbrace{\bigoplus_{n\in\integer}( F \to F )^{\oplus m_n} [n]}_{D:=},
\]
where $F\to F$ is an exact complex concentrated in degree 0 and 1, $m_n=\dim L^n$, and $[n]$ being the homological shift by $n$ times.  Now $C^{[w]\bullet}$ is given by
\begin{eqnarray*}
H^\bullet(C^\bullet)^{\otimes w}\otimes F\sn_w &\oplus& \left(\bigoplus_{i=1}^w D^{\otimes i-1}\otimes H^\bullet(C^\bullet) \otimes D^{\otimes w-i-1}\right)\otimes F\sn_w\\
& \oplus& \cdots \oplus\;\; D^{\otimes w}\otimes F\sn_w.
\end{eqnarray*}
Note that (the total complex of)  $((F \to F)^{\oplus m})^{\otimes k}$ has zero homology for any $m,k\geq 0$.  By considering the description of the differential on $C^{[w] \bullet}$, which is ``the same" as that for $C^{\otimes w\bullet}$, we have vanishing homology everywhere apart from $H^\bullet(C^\bullet)^{\otimes w}\otimes F\sn_w = H^\bullet(C^{\bullet})^{[w]}$.

If $f^\bullet: C^\bullet\to D^\bullet$ is a quasi-isomorphism, then we have vector space isomorphism
\[
H^\bullet(f^\bullet)=f^\bullet|_{H^\bullet(C^\bullet)} : H^\bullet(C^\bullet) \isomto H^\bullet(D^\bullet).
\]
Therefore, the wreath $f^{[w]\bullet}$ of $f^\bullet$ induces:
\begin{eqnarray*}
H^\bullet(f^{[w]\bullet}) = f^{[w]\bullet}|_{H^\bullet(C^{[w]\bullet})} = f^{[w]\bullet}|_{H^\bullet(C^\bullet)^{[w]}} = (f^\bullet|_{H^\bullet(C^\bullet)})^{[w]},
\end{eqnarray*}
which is an isomorphism from $H^\bullet(C^\bullet)^{[w]} = H^\bullet(C^{[w]\bullet})$ to $H^\bullet(D^\bullet)^{[w]} = H^\bullet(D^{[w]\bullet})$.
\end{proof}
\begin{remark}
Since $A$ is quasi-hereditary, $A$ has finite global dimension, and we have natural triangle equivalence $D^b(\gr{A})\simeq K^b(\gr{A})$.  The lemma can therefore be stated using derived category instead of homotopy category.
\end{remark}

The following three propositions are folklore that have not been shown in the literature explicitly to the best of our knowledge.
\begin{prop}\label{prop-wrExtCommute}
Wreathing and taking Ext-algebra of structural modules are commuting operations on algebras. i.e. $\Ext_{A^{[w]}}^\bullet (X^{[w]},X^{[w]})^\op \isom (A^X)^{[w]}$.
\end{prop}
\begin{proof}
Let $\projresX$ be the projective resolution of $X$.  We claim that we have the following algebra isomorphisms:
\begin{eqnarray*}
\Ext_{A^{[w]}}^\bullet (X^{[w]},X^{[w]}) &=& H^\bullet( \End_{A^{[w]}}(\projresX^{[w]}))\\
&\isom& H^\bullet( \End_A(\projresX)^{[w]} )\\
&\isom& H^\bullet( \End_A(\projresX) )^{[w]} \\
&=& (\Ext_A^\bullet(X,X))^{[w]}
\end{eqnarray*}
Note that the endomorphism rings above are taken over the dg algebra $A$ (dga concentrated in degree zero).  The third (algebra) isomorphism is justified by Lemma \ref{lem-wrHomology} above.

To justify and second isomorphism, we note that for any dg $A$-modules $M,N$, we have
\begin{eqnarray*}
& & \Hom_{A^{[w]}} (M^{\otimes w}\otimes k\sn_w , N^{\otimes w}\otimes k\sn_w )\\
 &\isom& \Hom_{A^{[w]}} ( A^{[w]}\otimes_{A^{\otimes w}} M^{\otimes w} , A^{[w]}\otimes_{A^{\otimes w}}N^{\otimes w} )\\
&\isom & \Hom_{A^{\otimes w}} ( M^{\otimes w} , \Hom_{A^{[w]}}(A^{[w]}, A^{[w]}\otimes_{A^{\otimes w}}N^{\otimes w} ))\\
&\isom & \Hom_{A^{\otimes w}} (M^{\otimes w} , \bigoplus_{\sigma\in\sn_w} \sigma\otimes N^{\otimes w})\\
&\isom & \bigoplus_{\sigma\in\sn_w}\sigma\otimes \Hom_{A^{\otimes w}}(M^{\otimes w},N^{\otimes w})\\
&\isom & A^{[w]}\otimes_{A^{\otimes w}} \Hom_A(M,N)^{\otimes w}\\
&\isom & \Hom_A(M,N)^{[w]}
\end{eqnarray*}
Therefore, we have vector space isomorphism $\End_{A^{[w]}}(\projresstd^{[w]})\isom \End_A(\projresstd)^{[w]}$, and this becomes an algebra isomorphism because $f^{[w]}g^{[w]}=(fg)^{[w]}$ by construction.
\end{proof}

\begin{lemma}\label{lem-dualOnWr}
Let $A$ be a BGG algebra, then its wreath product $A^{[w]}$ is also a BGG algebra.
\end{lemma}
\begin{proof}
The fact that $A^{[w]}$ is quasi-hereditary is the result of \cite[Section 6]{CT}.  For proving BGG duality on  $A^{[w]}$, similarly to Lemma \ref{lem-dualOnTensor}, suppose $\iota$ is the anti-automorphism of $A$ corresponding to the BGG duality on $A$, and $\iota_w$ the natural involutary anti-automorphism $\iota_w$ on the group algebra structure of $F\sn_w$, which sends $\sigma\in \sn_w$ to $\sigma^{-1}$.  These anti-automorphisms are compatible with the action of permuting the components of a tensor space, so they combine to give the anti-automorphism $\iota^{[w]}$ on $A^{[w]}$, which is given by
\begin{eqnarray*}
\iota^{[w]}(a_1\otimes \cdots \otimes a_w\otimes\sigma) = \iota(a_{\iota_w(\sigma)(1)})\otimes\cdots \otimes \iota(a_{\iota_w(\sigma)(n)}) \otimes \iota_w(\sigma).
\end{eqnarray*}
Indeed,
\begin{eqnarray*}
& & \iota^{[w]}( (a_1\otimes \cdots \otimes a_w\otimes \sigma)(b_1\otimes\cdots \otimes b_w\otimes \tau)) \\
&=& \iota^{[w]}( a_1b_{\sigma^{-1}(1)}\otimes \cdots\otimes a_wb_{\sigma^{-1}(w)}\otimes\sigma\tau) \\
&=& \iota(a_{\tau\sigma(1)}b_{\tau(1)})\otimes\cdots\otimes \iota(a_{\tau\sigma(w)}b_{\tau(w)})\otimes\tau^{-1}\sigma^{-1}\\
&=& \iota(b_{\tau(1)})\iota(a_{\tau\sigma(1)})\otimes\cdots\otimes \iota(b_{\tau(w)})\iota(a_{\tau\sigma(w)})\otimes\tau^{-1}\sigma^{-1}\\
&=& (\iota(b_{\tau(1)})\otimes\cdots\otimes \iota(b_{\tau(w)})\otimes \tau^{-1})(\iota(a_{\sigma(1)})\otimes\cdots \otimes\iota(a_{\sigma(1)})\otimes\sigma^{-1})\\
&=& \iota^{[w]}(b_1\otimes\cdots\otimes \tau)\iota^{[w]}(a_1\otimes\cdots\otimes \sigma)
\end{eqnarray*}
Note that as $\iota^2=\alpha$ is an inner automorphism of $A$, so $(\iota^{[w]})^2=\alpha^{\otimes w}\otimes 1_{\sn_w}$ is also an inner automorphism of $A^{[w]}$. Let the duality functor of $A$ be $\delta$ and the corresponding functor on $A^{[w]}$ be $\delta^{[w]}$.  Now the duality induced by $\iota^{[w]}$ maps any simple $A^{[w]}$-modules to its linear dual (c.f. Lemma \ref{lem-balOnTensor}), then the fact that $\delta^{[w]}$ fixes simples follows from Remark (3) of \cite[Coro 3.9]{CT}.
\end{proof}
\begin{remark}
Recall that the cell datum of a cellular algebra consists of an involutary anti-automorphism.  As cellular algebras have such close resemblance of BGG algebras, one may expect the same result to hold for cellular algebras as well.  A special case is already known, see \cite{GG}.
\end{remark}

\begin{prop}\label{prop-balOnWr}
If $A$ is a standard Koszul (resp. balanced) algebra, then so is $A^{[w]}$.
\end{prop}
\begin{proof}
For convenience, we prove this by assuming $A$ is basic, there is no loss in generality since all our stated facts respect Morita equivalence.  As $A= \bigoplus_{n\geq 0} A_n$ is Koszul,  $A_0$ admits a linear projective resolution $K^\bullet$.  Note that $A_0 \isom \bigoplus_{i\in I} L(i)$, and we have $A_0^{[w]} \isom \bigoplus_{\blambda \in \Lambda_w^I} L(\blambda)^{\oplus m(\blambda)}$.  From Lemma \ref{lem-wrHomology}, we know that $K^{[w]\bullet}$ is a linear projective complex whose homology is $A_0^{[w]}$ concentrated in (homological) degree 0, hence $K^{[w]\bullet}$ is a linear projective resolution of $A_0^{[w]}$.

If $A$ is balanced, then $A$ and its Ringel dual $A^T$ are Koszul.  We already have $A^{[w]}$ Koszul.  By (a Morita equivalent version of) Proposition \ref{prop-wrExtCommute}, the Ringel dual of $A^{[w]}$ is isomorphic to the basic algebra of $(A^T)^{[w]}$.  Since $A^T$ is Koszul, $(A^T)^{[w]}$ is also Koszul, and hence the Ringel dual of $A^{[w]}$ is Koszul.  This is sufficient for $A$ to be balanced.
\end{proof}
\begin{remark}
Alternatively, we can look at (co)resolutions of standard and costandard modules, i.e. working with the definition of standard Koszulity and balancedness directly.  Using the fact that wreathing preserves linearity, Lemma \ref{lem-wrHomology}, and some other folklore which we have not shown, (namely wreathing the direct sum of the structural (co)resolutions is the direct sum of the corresponding structural (co)resolutions), one can then show standard Koszulity and balancedness for the wreath product algebra.
\end{remark}

We finish by showing that wreathing also inherits condition (H).
\begin{theorem}
$A$ be a standard Koszul BGG algebra satisfying condition (H).  Then $A^{[w]}$ is also a standard Koszul BGG algebra satisfying condition (H).
\end{theorem}
\begin{proof}
By Lemma \ref{lem-dualOnWr} and Proposition \ref{prop-balOnWr}, it is sufficient to show that condition (H) is satisfied for $A^{[w]}$.  Let $h_A$ be the function so that $A$ satisfies condition (H).  We prove that condition (H) holds in $A^{[w]}$ using the following function:
\begin{eqnarray*}
h: \Lambda_w^I & \to & \{0,1,\ldots,n\}\\
 (x_1,\ldots,x_w) &\mapsto & \left(\sum_{k=1}^w h_A(x_k)|\lambda^{(k)}|\right) - (w-1)
\end{eqnarray*}

Suppose $[\std(\blambda):L(\bmu)\langle l\rangle]_q\neq 0$.  From \cite[Prop 4.4]{CT}, there exist $(\rho^{(i,s)})_{(i,s)\in K} \in \Lambda_w^K$ such that
\begin{eqnarray}
\sum_{s=0}^{m_i} |\rho^{(i,s)}| &=& |\lambda^{(i)}|\quad \forall i\in I, \label{cond1} \\
\sum_{(j,s)\in K_i} |\rho^{(j,s)}| &=& |\mu^{(i)}|\quad \forall i\in I, \label{cond2} \\
\sum_{(i,s)\in K} l_{i,s} |\rho^{(i,s)}| &=& l \label{cond3}
\end{eqnarray}
where
\begin{eqnarray*}
K &=& \{(i,s)\in I\times \integer \;|\;\; 0\leq s\leq m_i\},\\
K_i &=& \{(j,s)\in K \;|\;\; \std(j,s)/\std(j,s+1) \isom L(i)\},\\
\mbox{and for all }i\in I,\;\; \std(i) &=& \std(i,0) \supset \std(i,1)\supset \cdots \supset \std(i,m_i+1) = 0
\end{eqnarray*}
is a refinement of the radical filtration of $\std(i)$ such that each subquotient is a single simple $A$-module; and the $l_{i,s}$-th radical layer of $\std(i)$ contains $\std(i,s)/\std(i,s+1)$.

We can see that $K=\coprod_{i\in I} K_i$.  Indeed, by definition $K$ is the union of $K_i$, and every $(j,s)$ lies in $K_i$ for precisely one $i$ as $\std(j,s)/\std(j,s+1)$ is a single simple module by definition.  To prove that condition (H) holds, it suffices to show that $h(\blambda)-h(\bmu)=l$.  Expanding the left hand side using the definition and conditions (\ref{cond1}) and (\ref{cond2}), we get
\begin{eqnarray*}
h(\blambda)-h(\bmu) &=& \sum_{i\in I} h_A(i) |\lambda^{(i)}|-\sum_{i\in I} h_A(i) |\mu^{(i)}|\\
&=& \sum_{i\in I} h_A(i) \left( \sum_{s=0}^{m_i} |\rho^{(i,s)}|\right) - \sum_{i\in I} h_A(i) \left( \sum_{(j,s)\in K_i} |\rho^{(j,s)}|\right).
\end{eqnarray*}
Since $(j,s)\in K_i$ implies $\std(j,s)/\std(j,s+1)\isom L(i)$ and $L(i)$ is in the $l_{j,s}$-th radical layer of $\std(j)$, i.e. $[\std(j),L(i)\langle l_{j,s}\rangle]_q\neq 0$, by condition (H) on $A$, we get $h_A(i)=h_A(j)-l_{j,s}$.  We substitute this back into our expansion and get:
\begin{eqnarray*}
h(\blambda)-h(\bmu) &=& \sum_{i\in I}\sum_{s=0}^{m_i} h_A(i)|\rho^{(i,s)}| - \sum_{i\in I} \sum_{(j,s)\in K_i} (h_A(j)-l_{j,s}) |\rho^{(j,s)}|\\
&=& \sum_{(i,s)\in K} h_A(i)|\rho^{(i,s)}| - \sum_{(i,s)\in K}\left(h_A(i)|\rho^{(i,s)}| - l_{i,s}|\rho^{(i,s)}|\right)\\
&=& \sum_{(i,s)\in K} l_{i,s}|\rho^{(i,s)}|\\
&=& l \; ,
\end{eqnarray*}
where the second equality here uses $K=\coprod_{i\in I}K_i$ and the third one uses condition (\ref{cond3}).  This completes the proof.
\end{proof}
\begin{remark}
Note that the ending term $-(w-1)$ in the definition of $h$ is not necessary to prove that condition (H) holds; but it is convenient in practice to include such a ``shift", because if $\min_{i\in I}\{ h_A(i)\}=0$, then $\min_{\blambda\in \Lambda_w^I} h(\blambda) = 0$.
\end{remark}

\textsc{Acknowledgements}  I thank Will Turner for correction on Proof of Proposition \ref{prop-wrExtCommute}; I also thank the referee, Abbie Hall, and Amy Pang for general comments and grammatical corrections.

\phantomsection

\small{
Address:  Institute of Mathematics, Fraser Noble Building, University of Aberdeen, Aberdeen, United Kingdom, AB24 3UE\\
E-mail:  \texttt{r01akyc@abdn.ac.uk}}

\end{document}